\documentclass[a4paper,12pt,oneside,reqno]{amsart}
\usepackage{amsfonts,amsmath, amssymb,amsthm}
\usepackage{graphicx,psfrag}
\usepackage[latin1]{inputenc}
\usepackage{cite}

\DeclareMathOperator*{\GUE}{GUE}
\DeclareMathOperator*{\GOE}{GOE}
\DeclareMathOperator*{\Airy}{Airy}

\newcommand{\PP}{\ensuremath{\mathbb{P}}}
\newcommand{\N}{\ensuremath{\mathbb{N}}}
\newcommand{\R}{\ensuremath{\mathbb{R}}}

\newcommand{\Z}{\ensuremath{\mathbb{Z}}}

\newcommand{\A}{\ensuremath{\mathcal{A}_2}}

\newcommand{\fr}{\ensuremath{\mathfrak{F}}}
\newcommand{\hl}{\ensuremath{\mathcal{HL}}}
\newcommand{\pp}{\ensuremath{\mathcal{PP}}}

\newtheorem{theorem}{Theorem}[section]

\newtheorem{lemma}[theorem]{Lemma}

\newtheorem{rem}[theorem]{Remark}
\newenvironment{remark}{\begin{rem}\normalfont}{\end{rem}}

\title{Universality of slow decorrelation in KPZ growth}

\author[I. Corwin]{Ivan Corwin}
\address{I. Corwin\\
  Courant Institute of Mathematical Sciences\\
  New York University\\
  251 Mercer Street\\
  New York, NY 10012, USA}
\email{corwin@cims.nyu.edu}

\author[P.L. Ferrari]{Patrik L. Ferrari}
\address{P.L. Ferrari\\
  Institute of Applied Mathematics\\
  University of Bonn\\
  Endenicher Allee 60\\
  53115 Bonn, Germany}
\email{ferrari@uni-bonn.de}

\author[S. P\'{e}ch\'{e}]{Sandrine P\'{e}ch\'{e}}
\address{S. P\'{e}ch\'{e}\\
  Institut Fourier\\
  100 Rue des maths\\
  38402 Saint Martin d'Heres, France}
\email{Sandrine.Peche@ujf-grenoble.fr}

\subjclass[82C22,60K35]{82C22, 60K35}

\date{23. February 2011}
\begin{document}
\sloppy\maketitle

\begin{abstract}
There has been much success in describing the limiting spatial fluctuations of growth models in the Kardar-Parisi-Zhang (KPZ) universality class. A proper rescaling of time should introduce a non-trivial temporal dimension to these limiting fluctuations. 
In one-dimension, the KPZ class has the dynamical scaling exponent $z=3/2$, that means one should find a universal space-time limiting process under the scaling of time as $t\,T$, space like $t^{2/3} X$ and fluctuations like $t^{1/3}$  as $t\to\infty$.

In this paper we provide evidence for this belief. We prove that under certain hypotheses, growth models display temporal slow decorrelation. That is to say that in the scalings above, the limiting spatial process for times $t\, T$ and $t\, T+t^{\nu}$ are identical, for any $\nu<1$. The hypotheses are known to be satisfied for certain last passage percolation models, the polynuclear growth model, and the totally / partially asymmetric simple exclusion process. Using slow decorrelation we may extend known fluctuation limit results to space-time regions where correlation functions are unknown.

The approach we develop requires the minimal expected hypotheses for slow decorrelation to hold and provides a simple and intuitive proof which applied to a wide variety of models.
\end{abstract}

\section{Introduction}
Kardar, Parisi and Zhang (KPZ)~\cite{KPZ:1986d} proposed on physical grounds that a wide variety of irreversible stochastically growing interfaces should be governed 
by a single stochastic PDE (with two model dependent parameters $D,\lambda\neq 0$). Namely, let $x\mapsto h(x,t)\in\R$ be the height function at time $t$ and position $x\in\R^d$, then the KPZ equation is
\begin{equation*}
\frac{\partial h(x,t)}{\partial t} = D \Delta h(x,t) + \lambda|\nabla h(x,t)|^2 +\eta(x,t),
\end{equation*}
where $\eta(x,t)$ is a local noise term modeled by space-time white noise. Since then, it has been of significant interest to make mathematical sense of this SPDE (which is ill-posed due to the non-linearity) and to find the solutions for large growth time $t$.

Significant progress has been made towards understanding this equation in the one-dimensional $d=1$ case. Specifically, it is believed that the 
dynamical scaling exponent is $z=3/2$. This should mean that for any growth model (also polymer models) in the same universality class as the KPZ equation (i.e., the KPZ universality class), after centering $h$ by its asymptotic value
$\bar{h}(v):=\lim_{t\to\infty} \frac1t h(vt,t)$
and rescaling\footnote{Here we use $t$ as large parameter. In the literature also the choice $\epsilon^{-z}=t$ and $\epsilon\to 0$ is used.}
\begin{equation}\label{fixedpt}
h_{t}(X,T)= \frac{h(v T t+X (T t)^{2/3},T t)-T t\, \bar{h}(v+ X (T t)^{-1/3})}{(T t)^{1/3}},
\end{equation}
the limit of $h_{t}(X,T)$ should exist (as $t\to\infty$) and be independent of $v$. 
Moreover, the limit should, regardless of microscopic differences in the original models, converge to the same space-time 
process \footnote{As explained in the forthcoming paper~\cite{CQ:2011r}, this space-time process is expected to be a 
non-trivial renormalization fixed point for the whole KPZ universality class. See also \cite{Krug99,Spo95,Maj97} for previous 
discussion of space-time scalings in the physics literature.}.

Most of the rigorous work\footnote{Results for joint distributions at different times before taking $t$ to infinity have been derived, first in the problem of tagged particle in the TASEP~\cite{SI07}, and then in more general models~\cite{BF:2008l,BFS:2008l}. However, the different times are restricted to lie in an interval of width $O(t^{2/3})$.} done in studying the statistics associated with this fixed point have dealt with the spatial process (obtained as the asymptotic statistics of $h_{t}(X,T=1)$ as a process in $X$, as $t\to\infty$) and not on how the spatial process evolves with $T$. The exact form of these statistics depend only on the type of initial geometry of the growth process (e.g., the $\Airy_1$ process for non-random flat geometries and $\Airy_2$ process for wedge geometries; see the review~\cite{Fer10b}).

Computations of exact statistics require a level of solvability and thus have only been proved in the context of certain solvable discrete growth models or polymer models in the KPZ universality class. The partially/totally asymmetric simple exclusion process (P/TASEP), last passage percolation (LPP) with exponential or geometric weights, the corner growth model, and polynuclear growth (PNG) model constitute those models for which rigorous spatial fluctuation results have been proved. Recently, progress was made on analyzing the solution of the KPZ equation itself \cite{ACQ:2010p,SS:2010u, CQ:2010u,BG:1997s}, though this still relied on the approximation of the KPZ equation by a solvable discrete model.

The slow decorrelation phenomenon provides one of the strongest pieces of evidence that the above scaling is correct. Indeed, slow decorrelation means that $h_{t}(X,T) - h_{t}(X,T+t^{\nu-1})$ converges to zero in probability for any $\nu<1$. Fix $m$ times of the form $Tt + \alpha_i t^{\nu}$  (for $\alpha_i\in \R$ and $0<i\leq m$). Then, as long as $\nu<1$, the height function fluctuations, scaled by $t^{1/3}$ and considered in a spatial scale of $t^{2/3}$, will be asympotically (as $t\rightarrow \infty$) the same as those at time $Tt$.

Specifically, we introduce a generalized LPP model which encompasses several KPZ class models. Then we give sufficient conditions under which such LPP models display slow decorrelation. These conditions (the existence of a limit shape and one-point fluctuation result) are very elementary and 
hold for all the solvable models already mentioned, and are believed to hold for all KPZ class models. The proof that slow decorrelation 
follows from these two conditions is very simple -- it relies on the superadditivity property of LPP and on the simple observation that if 
$X_t\geq Y_t$ and both $X_t$ and $Y_t$ converge in law to the same random variable, then $X_t-Y_t$ converges in probability to zero (see Lemma~\ref{BAC_lemma}).

Previously, the slow decorrelation phenomenon was proved for the PNG model~\cite{PLF:2008s}. Therein the proof is based on very sharp estimates known in the literature only for the PNG. Apart from the PNG, the only other model for which slow decorrelation has been proved is TASEP under the assumption of stationary initial distribution~\cite{BFP:2009l}.

Besides being of conceptual interest, the slow decorrelation phenomenon is an important technical tool that allows one to, for instance: (a) easily translate limit process results between different related observables (e.g., total current, height function representation, particle positions in TASEP; see~\cite{BFP:2009l}), and more importantly, (b) prove limit theorems beyond the situations where the correlation functions are known \cite{SI07,BF:2008l,BFS:2008l} (see Section~\ref{corner_growth_model_sec}). A further application is in extending known process limit results to prove similar results for more general initial conditions / boundary conditions \cite{CFP:2009l}.

\subsection*{Outline}
In Section~\ref{gen_theory_sec} we introduce the general framework for LPP models in which we prove a set of criteria for slow decorrelation (Theorem~~\ref{growth_thm}). In the rest of the paper, we apply Theorem~\ref{growth_thm} to various models in the KPZ class, which can be related in some way with a LPP model: the corner growth model, point to point and point to line LPP models, TASEP, PASEP (which requires a slightly different argument since it cannot be directly mapped to a LPP problem) and PNG models. Finally we note extensions of the theorem to first passage percolation and directed polymers, provided that (as conjectured) the same criteria are satisfied.

\subsection*{Acknowledgments}
The authors wish to thank Jinho Baik for early discussions about this and related problems. 
I. Corwin wishes to thank the organizers of the ``Random Maps and Graphs on Surfaces'' conference at the Institut Henri Poincar\'{e}, 
as much of this work was done during that stay. Travel to that conference was provided through the PIRE grant OISE-07-30136 for which 
thanks goes to Charles Newman and G\'{e}rard Ben Arous for arranging for this funding. I. Corwin is funded by the NSF Graduate 
Research Fellowship. S. P\'{e}ch\'{e} would like to thank Herv\'{e} Guiol for useful discussions on TASEP and 
her work is partially supported by the Agence Nationale de la Recherche grant \mbox{ANR-08-BLAN-0311-01}.
The authors are very grateful to the anonymous referee for careful reading and a number of constructive remarks.

\section{A sufficient condition for slow decorrelation} \label{gen_theory_sec}
In this section we consider a general class of last passage percolation models (or equivalently growth models). Given the existence of a law of large numbers (LLN) and central limit theorem (CLT) for last passage time (or for the associated height function), we prove that such models display slow decorrelation along 
their specific, model dependent, ``characteristic'' directions.

We consider growth models in $\R^{d+1}$ for $d\geq 1$ which may be lattice based or driven by Poisson point processes. We define a directed LPP model to be an almost surely sigma-finite random non-negative measure $\mu$ on $\R^{d+1}$. For example we could take $\mu$ to be a collection of delta masses at every point of $\Z^{d+1}$ with weights given by random variables (which need not be independent or identically distributed).
Alternatively we could have a Poisson point process such as in the LPP realization of the PNG model. We will focus on a statistic we call the {\it directed half-line to point last passage time}. We choose to study this since, by specifying different distributions on the random measure $\mu$ one can recover statistics for a variety of KPZ class models. In order to define this passage time we introduce the half-line
\begin{equation*}
\hl=\{p:p_1=p_2=\cdots =p_{d+1} \leq 0\},
\end{equation*}
where $p_i$ is the $i$ coordinate of the point $p$.
\begin{figure}
\begin{center}
\psfrag{x}[c]{$x$}
\psfrag{t}[l]{$t$}
\psfrag{p}[l]{$p$}
\psfrag{pi}[l]{$\pi$}
\psfrag{HL}[c]{$\mathcal{HL}$}
\includegraphics[height=4.5cm]{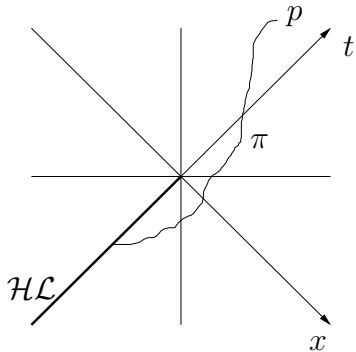}
\caption{The black line is half-line $\hl$ and the space and time axis are label. A directed path $\pi$ from $\hl$ to the point $p$ is shown.}
\label{slow_dec_space_time}
\end{center}
\end{figure}

It is convenient for us to define a second coordinate system which we call the space-time coordinate system as follows: Let $R$ be the rotation matrix which takes $\hl$ to $\{p:p_1\leq 0, p_2=\cdots=p_{d+1}=0\}$. Then the space-time coordinate system is $R^{-1}$ applied to the standard basis. The line $\{p:p_1=p_2=\cdots=p_{d+1}\}$ (which contains $\hl$) is the inverse image of $\{p:p_2=\cdots=p_{d+1}=0\}$ and we call it the $t$-axis (for ``time''), see Figure~\ref{slow_dec_space_time} for an illustration. The other space-time axes are labeled $x_1$ through $x_d$ (these are considered to be ``space'' axes). Call a curve $\pi$ in $\R^{d+1}$ a directed path if $\gamma=R\pi$ is a function of $t$ and is 1-Lipschitz. Two points are called ``time-like'' if they can be connected by such a path. Otherwise they are called ``space-like''.

To a directed path we assign a passage time
\begin{equation*}
T(\pi) = \mu(\pi)
\end{equation*}
which is the measure, under the random measure $\mu$ of the curve $\pi$.
Now we define the last passage time from the half-line $\hl$ to a point $p$ as
\begin{equation*}
L_{\hl}(p)=\sup_{\pi:\hl\to p} T(\pi),
\end{equation*}
where we understand the supremum as being over all directed paths starting from the half-line and going to $p$. One may also consider point to point last passage time between $p$ and $q$ which we write as $L_{\pp}(p,q)$. This is the special case of $\mu\equiv 0$ on $\{x: x-p \in\R^{d+1}\setminus \R_+^{d+1}\}$.

In Section~\ref{apps} we show how, by specifying the random measure $\mu$ differently, this model encompasses a wide variety of LPP models and related processes (such as TASEP and PNG). Just to illustrate though, take $d=1$ and let $\mu$ be composed of only delta masses at points $p$ in $\Z^2_+$ with mass $w_{p}$ exponentially distributed with rate 1. Then $L_{\hl}(p)$ is the last passage time for the usual LPP in a corner (or equivalently the corner growth model considered in Section~\ref{corner_growth_model_sec}).
We present our result in this more general framework to allow for non-lattice models such as the PNG model.

We can now state a result showing that \emph{slow decorrelation occurs} in any model which can be phrased in terms of this type of last passage percolation model \emph{provided both a LLN and a CLT hold}.

\begin{theorem}\label{growth_thm}
Fix a last passage model in dimension $d+1$ with $d\geq 1$ by specifying the distributions of the random variables which make up the environment.
Consider a point $p\in \R^{d+1}$ and a time-like direction $u\in \R_+^{d+1}$.
If there exist constants (depending on $p$, $u$, and the model): $\ell_{\hl}$ and $\ell_{\pp}$ non-negative;
$\gamma_{\hl},\gamma_{\pp}\in (0,1)$; $\nu\in(0,\gamma_{\hl}/\gamma_{\pp})$; distributions $D$, $D'$;
and scaling constants $c_{\hl},c_{\pp}$ such that
\begin{equation*}
\begin{aligned}
\chi_1(t)&:=\frac{L_{\hl}(tp)-t\ell_{\hl}}{c_{\hl}t^{\gamma_{\hl}}}\Longrightarrow D, \quad \textrm{as }t\textrm{ goes to infinity},\\
\chi_2(t)&:=\frac{L_{\hl}(tp+t^{\nu}u)-t\ell_{\hl}-t^\nu \ell_{\pp}}{c_{\hl}t^{\gamma_{\hl}}}\Longrightarrow D, \quad \textrm{as }t\textrm{ goes to infinity},\\
\chi_3(t)&:=\frac{L_{\pp}(tp,tp+t^{\nu}u)-t^\nu \ell_{\pp}}{c_{\pp}(t^{\nu})^{\gamma_{\pp}}}\Longrightarrow D', \quad \textrm{as }t\textrm{ goes to infinity},
\end{aligned}
\end{equation*}
then we have slow decorrelation of the half-line to point last passage time at $tp$, in the direction $u$ and with scaling exponent $\nu$,
which is to say that for all  $M>0$,
\begin{equation}\label{slow_dec}
\lim_{t\to \infty} \PP(|L_{\hl}(tp+t^{\nu}u)-L_{\hl}(tp)-t^\nu \ell_{\pp}|\geq M t^{\gamma_{\hl}})=0.
\end{equation}

\end{theorem}

\begin{remark}\label{non_const_remark}
There are many generalizations of this result whose proofs are very similar. For instance the fixed (macroscopic) point $p$ and the fixed direction $u$ can, in fact, vary with $t$ as long as they converge as $t\to \infty$. One may also think of the random LPP measure $\mu$ (and the associated probability space $\Omega$) as depending on $t$. Thus for each $t$ the LPP environment is given by $\mu_t$ defined on the space $\Omega_t$. The probability $\PP$ will therefore also depend on $t$, however an inspection of the proof below shows that the whole theorem still holds with $\PP$ replaced by $\PP_t$.
\end{remark}

\begin{proof}[Proof of Theorem~\ref{growth_thm}]
Recall the super-additivity property:
\begin{equation*}
L_{\hl}(tp+t^{\nu}u)\geq L_{\hl}(tp)+L_{\pp}(tp,tp+t^{\nu}u),
\end{equation*}
which holds provided the last passage times are defined on the same probability space.
This follows from the fact that, by restricting the set of paths which contribute to $L_{\hl}(tp+t^{\nu}u)$ to only those which go through the point $tp$, one can only decrease the last passage time. The following lemma plays a central role in our proof.

\begin{lemma}[Lemma~4.1 of~\cite{BAC:2009c}]\label{BAC_lemma}
Consider two sequences of random variables $\{X_n\}$ and $\{\tilde{X}_n\}$ such that for each $n$,
$X_n$ and $\tilde{X}_n$ are defined on the same probability space $\Omega_n$. If $X_n\geq\tilde{X}_n$ and $X_n\Rightarrow D$
as well as $\tilde{X}_n\Rightarrow D$ then $X_n-\tilde{X}_n$ converges to zero in probability.
Conversely if $\tilde{X}_n\Rightarrow D$ and $X_n-\tilde{X}_n$ converges to zero in probability then $X_n\Rightarrow D$ as well.
\end{lemma}
From now on, we assume that the different last passage times $L_{\hl}(\cdot) $ and $L_{\pp}(\cdot)$ are realized on the same probability space. Also, by absorbing the constants $c_{\hl}$ and $c_{\pp}$ into the distributions, we may fix them to be equal to one. Using super-additivity we may write
\begin{equation}\label{growth_compensator}
L_{\hl}(tp+t^{\nu}u)=L_{\hl}(tp)+L_{\pp}(tp,tp+t^{\nu}u)+X_t,
\end{equation}
where $X_t\geq 0$ is a (compensator) random variable. Rewriting the above equation in terms of the random variables $\chi_1(t)$, $\chi_2(t)$ and $\chi_3(t)$ and dividing by $t^{\gamma_{\hl}}$ we are left with
\begin{equation*}
 \chi_2(t) = \chi_1(t) + \chi_3(t) t^{\nu\gamma_{\pp}-\gamma_{\hl}} +X_t t^{-\gamma_{\hl}}.
\end{equation*}
By assumption on $\nu$, $\nu\gamma_{\pp}-\gamma_{\hl}<0$ and
hence we know that the distribution of $\chi_2(t)$ and separately of $\chi_1(t) + \chi_3(t) t^{\nu\gamma_{\pp}-\gamma_{\hl}}$
converge to the same distribution $D$.
However, since $X_t t^{-\gamma_{\hl}}$ is always non-negative, we also know that
$\chi_2(t) \geq \chi_1(t) + \chi_3(t) t^{\nu\gamma_{\pp}-\gamma_{\hl}}$.
Therefore, by Lemma~\ref{BAC_lemma} their difference, $X_t t^{-\gamma_{\hl}}$, converges to zero in probability.
Thus $\chi_2(t)-\chi_1(t)$ converges to zero in probability. Since
\begin{equation*}
\chi_2(t)-\chi_1(t) = \frac{L_{\hl}(tp+t^{\nu}u)-L_{\hl}(tp)-t^\nu \ell_{\pp}}{t^{\gamma_{\hl}}}
\end{equation*}
the theorem immediately follows.
\end{proof}

\section{Slow decorrelation in KPZ growth models:\\ examples and applications}\label{apps}

The aim of this section is to make a non-exhaustive review of the possible fields of applications of Theorem~\ref{growth_thm}. We introduce a few standard models and explain, briefly, how they fit into the framework of half-line to point LPP and what the consequences of Theorem~\ref{growth_thm} are for these models.

\subsection{Corner growth model}\label{corner_growth_model_sec}
We choose to first develop in detail the implications of slow decorrelation for a simple KPZ growth process. This process is known as the {\it corner growth model} and is related to both LPP and TASEP .

Consider a set $A_t$ in $\R_+^2$ with initial condition $A_0=\R_+^2$ and evolving under the following dynamics: from each {\it outer corner} of $A_t$ a $[0,1)\times[0,1)$-box is fill at rate one (i.e., after exponentially distributed waiting time of mean $1$). See Figure~\ref{corner_fig} for an illustration of this growth rule where the model has been rotated by $\pi/4$.
\begin{figure}
\begin{center}
\psfrag{x}[cb]{$X$}
\psfrag{i}[c]{$x$}
\psfrag{j}[c]{$y$}
\psfrag{w11}[c]{$w_{1,1}$}
\psfrag{w12}[c]{$w_{1,2}$}
\psfrag{w13}[c]{$w_{1,3}$}
\psfrag{w21}[c]{$w_{2,1}$}
\psfrag{w22}[c]{$w_{2,2}$}
\psfrag{w31}[c]{$w_{3,1}$}
\includegraphics[height=5cm]{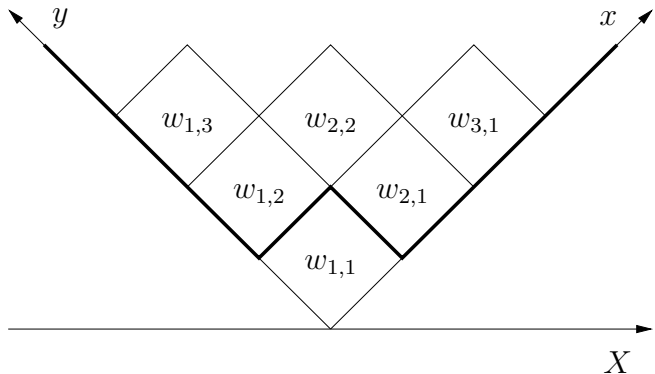}
\caption{Corner growth height function. The $w_{i,j}$ are the random, exponentially distributed times it takes to fill an outer corner. The black line is the height function at time $t=w_{1,1}$. There are two outer corners at that time.}
\label{corner_fig}
\end{center}
\end{figure}
One can record the evolution of the growing interface $\partial A_t$ in terms of the random variable
\begin{equation*}
L(x,y):=\inf\{ t\geq 0 | (x-\tfrac12,y-\tfrac12)\not\in A_t\}\textrm{ for }(x,y)\in \Z_+^2.
\end{equation*}

This random variable is well known to define a last passage time in a related directed percolation model, as we now recall.
Let $w_{i,j}$ be the waiting time for the outer corner $(i,j)$ to be filled, once it appears. A path $\pi$ from $(1,1)$ to $(x,y)$ is called directed if it moves either up or to the right along lattice edges from $(1,1)$ to $(x,y)$. To each such path $\pi$, one associates a passage time $T(\pi)=\sum_{(i,j)\in\pi} w_{i,j}$. Then
\begin{equation*}
L(x,y)=\max_{\pi:(1,1)\to (x,y)} T(\pi),
\end{equation*} i.e., $L$ is the last passage times from $(1,1)$ to $(x,y)$.

Alternatively one can keep track of $A_t$ in terms of a height function $h$ defined by the relationship
\begin{equation}\label{height_lpp}
\{h(x-y,t)\geq x+y\}= \{L(x,y)\leq t\}\quad (x,y)\in\Z_+^2,
\end{equation}
(together with linear interpolation for non-integer values of $x-y$). Note that for given $t$, $h(X,t)=|X|$ for $|X|$ large enough.
Thus the corner growth process is equivalent to the stochastic evolution of a height function $h(X,t)$ with $h(X,0)=|X|$ and growing according to the rule that local valleys ($\diagdown\;\!\!\diagup$) are replaced by local hills ($\diagup\;\!\!\diagdown$) at rate one. We will speak mostly about the height function, but when it comes to computing and proving theorems it is often easier to deal with the last passage picture.

\subsubsection{LLN and CLT}
Analogous to the LLN for sums of i.i.d.\ random variables, there exists an almost sure limit shape for the height function~\cite{HR:1981n,TS:1998h} of this growth model:
\begin{equation*}
\bar{h}(v):=\lim_{t\to \infty} \frac{h(vt,t)}{t}=
\begin{cases}
\frac{1}{2}(v^2+1), & \textrm{ for }v\in (-1,1),\\
|v|, & \textrm{ for }v\notin (-1,1).
\end{cases}
\end{equation*}

If we more generally consider the height function arising from a LPP model with an ergodic measure (e.g., iid random lattice weights), then super-additivity and the Kingman's ergodic theorem implies the existence of a (possibly infinity) limit growth evolution
\begin{equation*}
\tilde{h}(X,T) := \lim_{t\rightarrow \infty} \frac{h(Xt,Tt)}{t}.
\end{equation*}
Since LPP is a variation problem, the limiting profile is also given by the solution to a variational problem which, when translated into the limiting height function means that $\tilde{h}$ satisfies a Hamilton-Jacobi PDE $\partial_T \tilde{h} = f(\partial_X \tilde{h})$ for a model dependent flux function $f$. Such PDEs may have multiple solutions, and $\tilde{h}$ corresponds to the 
unique weak solutions subject to entropy conditions. Such PDEs can be solved via the method of characteristics~\cite{E:1998p}. Characteristics are lines of slope $f'$ along which initial data for $\tilde{h}$ is transported. In our present case if we set $\rho = \tfrac{1}{2}(1-\partial_X \tilde{h})$ then $\rho$ satisfies the Burgers equation $\partial_T \rho =\partial_X(\rho(1-\rho))$
and the characteristic lines are of constant velocity emanating out of the origin.

It is the fluctuations around this macroscopic profile which are believed to be universal. Johansson~\cite{KJ:2000s} proved that asymptotic one-point fluctuations are given by
\begin{equation}\label{eq8}
\lim_{t\to \infty} \PP\left(\frac{h(vt,t)-t \bar{h}(v)}{2^{-1/3} (1-v^2)^{2/3} t^{1/3}} \geq s\right) = F_{\GUE}(s),
\end{equation}
where $F_{\GUE}$ is the $\GUE$ Tracy-Widom distribution defined in \cite{TW:1994l}.
Unlike the traditional CLT the fluctuations here are in the order of $t^{1/3}$ and the limiting distribution is not Gaussian.

Likewise we may consider the fluctuations at multiple spatial locations by fixing
\begin{equation*}
\begin{aligned}
 X(\tau) &= vt +\tau(2(1-v^2))^{1/3} t^{2/3},\\
 H(\tau,s) &= \frac{1+v^2}{2}t + \tau\, v(2(1-v^2))^{1/3} t^{2/3} +(\tau^2-s)\frac{(1-v^2)^{2/3}}{2^{1/3}} t^{1/3}.
\end{aligned}
\end{equation*}
Here $H(\tau,0)=t\, \bar{h}(X(\tau)/t)+o(1)$ and $H(\tau,s)-H(\tau,0)$ measures the fluctuations with respect to the limit shape behavior. Then, in the large time limit, the joint-distributions of the fluctuations are governed by the so-called $\Airy_2$ process, denoted by $\A$. This process was introduced by Pr\"{a}hofer and Spohn~\cite{PS:2002s} in the context of the PNG model (see also \cite{KJ:2003d}). A complete definition of the Airy$_2$ process is recalled in \cite{CFP:2009l}. More precisely, it holds that
\begin{equation}\label{theta_zero_thm}
 \lim_{t\to \infty} \PP\left(\bigcap_{k=1}^{m} \{h(X(\tau_k),t)\geq H(\tau_k,s_k)\}\right)
= \PP\left(\bigcap_{k=1}^{m}\{\A(\tau_k)\leq s_k\}\right),
\end{equation}
where $m \geq 1$, $\tau_1<\tau_2<\cdots <\tau_m$ and $s_1,\ldots, s_m$ are real numbers. Of course, (\ref{eq8}) is the special case of (\ref{theta_zero_thm}) for $m=1$ and $\tau_1=0$.

\subsubsection{Slow decorrelation in the corner growth model}
We now consider how fluctuations in the height function are carried through time.
For instance, if the fluctuation of the height function above the origin is known at time $t$ (large) for how long can we expect to see this fluctuation persist (in the $t^{1/3}$ scale)? The answer is non-trivial and given by applying Theorem~\ref{growth_thm}: there exists a single direction in space-time along which the height function fluctuations are carried over time scales of order $t^1$, while for all other directions only at space-time distances of order $t^{2/3}$.
Indeed, given a fixed velocity \mbox{$v\in (-1,1)$}, any exponent $\nu<1$, and any real number $\theta$,
\begin{equation}\label{eq11}
\lim_{t\to \infty} \PP\left(\left|[h\big(v(t+\theta t^{\nu}),t+\theta t^{\nu}\big)-(t+\theta t^{\nu})\bar{h}(v)]-[h\big(vt,t\big) - t\bar{h}(v)]
\right| \geq M t^{1/3}\right) = 0,
\end{equation}
for any $M>0$. Thus, the height fluctuations at time $t$ at position $vt$ and at time and $t+\theta t^\nu$ at position $v(t+\theta t^\nu)$ differs only of $o(t^{1/3})$. These fixed velocity space-time lines are the characteristic of the Burgers equation above.

Thus, the right space-time scaling limit to consider is that given in equation (\ref{fixedpt}), with $h$ and $\bar{h}$ now taken to be the TASEP height function and asymptotic shape, rather than that of the KPZ equation. As noted below equation (\ref{fixedpt}), the value of the velocity $v$ should not affect the law of the limiting space-time process. As evidence for this, equation (\ref{theta_zero_thm}) shows that we encounter the $\mathrm{Airy}_2$ process as a scaling limit regardless of the value of $v\in (-1,1)$. This amounts to saying that the fixed $T$ marginals of the full space-time limit process of equation (\ref{fixedpt}) are independent of $v$.

\subsubsection{Implications of slow decorrelation}
Up to now only the ``spatial-like behavior'' of the space-time process (\ref{fixedpt}), i.e., the process in the variable $x$ for fixed $\beta$ (which one can set to $1$ w.l.o.g.) has been obtained, while the process in the variable $\beta$ remains to be unraveled.

A consequence of (\ref{eq11}) is that if we look at the fluctuations at two moments of time $t$ and $t'$ with $|t'-t|\sim t^\nu$ with $\nu<1$, it corresponds to taking $\beta=1+\mathcal{O}(t^{\nu-1})$ in the r.h.s.\ of (\ref{fixedpt}). Then in the $t\to\infty$ limit, the limit process is identical to the process for fixed $\beta=1$. So, if we determine the limit process for any space-time cut such that in the $t\to\infty$ limit, $\beta\to 1$, then, thanks to slow decorrelation, one can extend the result to any other space-time cut with the same property. In the following we refer to this property as {\it the process limit extends to general $1+0$ dimensional space-time directions}, meaning that we have $1$ dimension with spatial-like behavior and $0$ dimensions in the orthogonal direction.

As indicated in the Introduction, slow decorrelation also allows for instance (a) to translate the limit of different related observables and (b) to extend results on fluctuation statistics to space-time regions where correlation functions are unknown. We illustrate these features in the context of the corner growth model. 
For simplicity, we consider the case where $v=0$. Fix $m\geq 1$, $\nu\in[0,1)$, real numbers $\tau_1<\tau_2<\cdots <\tau_m$ and $s_1,\ldots, s_m$. Then set
\begin{equation}\label{eq13}
\begin{aligned}
x(\tau,\theta) &= \lfloor \tfrac{1}{4}(t+\theta t^{\nu}) +\tau 2^{-2/3} t^{2/3}\rfloor,\\
y(\tau,\theta) &= \lfloor \tfrac{1}{4}(t+\theta t^{\nu}) -\tau 2^{-2/3} t^{2/3}\rfloor,\\
\ell(\tau,\theta,s) &= (t+\theta t^{\nu}) + (s-\tau^2) 2^{2/3} t^{1/3}.
\end{aligned}
\end{equation}

(a) We first show that one can recover (\ref{theta_zero_thm}) from an analoguous statement in the corresponding LPP model using (\ref{height_lpp}) and slow decorrelation.
We start from a result in LPP. Consider the fixed $y=t/4$ slice of space-time. This is obtained by setting $\theta_k t^\nu=\tau_k 2^{4/3}t^{2/3}$ in (\ref{eq13}), for which
\begin{equation}\label{eq14}
x(\tau,\theta)=\tfrac14 t + \tau 2^{1/3} t^{2/3},\quad y(\tau,\theta)=\tfrac14 t,\quad \ell(\tau,\theta,s)=t+\tau 2^{4/3} t^{2/3}+(s-\tau^2)2^{2/3} t^{1/3}.
\end{equation}
Using the Schur process~\cite{BP:2008a}, it is proven~\cite{SI07,CFP:2009l} that for $\theta_k t^\nu=\tau_k 2^{4/3}t^{2/3}$, $k=1,\ldots,m $,
\begin{equation}\label{lpp_limit_thm_eqn}
\lim_{t\to \infty} \PP\left(\bigcap_{k=1}^{m} \{L(x(\tau_k,\theta_k),\tfrac14 t)\leq \ell(\tau_k,s_k)\}\right)
 =\PP\left(\bigcap_{k=1}^{m}\{\A(\tau_k)\leq s_k\}\right).
\end{equation}
To get the result on the height function (\ref{theta_zero_thm}) using (\ref{height_lpp}), one would need to make the choice:
$\theta_k t^{\nu}=-(s_k-\tau_k^2) 2^{2/3} t^{1/3}$. Then we have
\begin{equation}\label{eq16}
x(\tau,\theta)-y(\tau,\theta) = \tau 2^{1/3} t^{2/3},\quad x(\tau,\theta)+y(\tau,\theta)
= \tfrac12 t-(s-\tau^2) 2^{-1/3} t^{1/3},\quad \ell(\tau,\theta,s)=t.
\end{equation}
Thus, to obtain (\ref{theta_zero_thm}) (for $v=0$) from (\ref{lpp_limit_thm_eqn}) it is actually sufficient to project $(x,y)$ in (\ref{eq16}) on the line $y=t/4$ along the 
characteristic line passing through $(x,y)$, see Figure~\ref{FigureDPPext} for an illustration. One finds that this projection gives the scaling (\ref{eq14}) but with $\tau$ replaced by some
$\tilde \tau=\tau (1+o(1))\to \tau$ as $t\to\infty$. The reason is that the characteristics for $\tau\neq 0$ have slope slightly different from $0$. Finally, slow decorrelation (Theorem~\ref{growth_thm}, see also (\ref{eq11})) and the union bound imply then that
\begin{equation*}
 \textrm{l.h.s.\ of }(\ref{theta_zero_thm})\big|_{v=0} = \textrm{l.h.s.\ of }(\ref{lpp_limit_thm_eqn})
=\PP\left(\bigcap_{k=1}^{m}\{\A(\tau_k)\leq s_k\}\right).
\end{equation*}
\begin{figure}[t!]
\begin{center}
\psfrag{tnu}[cb]{$\mathcal{O}(T^\nu)$}
\psfrag{x}[cb]{$x$}
\psfrag{y}[c]{$y$}
\includegraphics[height=5cm]{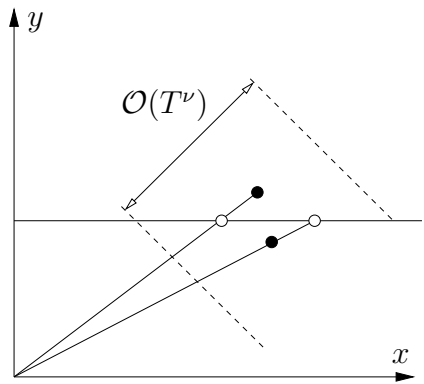}
\caption{Assume that the black dots are $\mathcal{O}(T^\nu)$ for some $\nu<1$ away from the line $y=t/4$.
Then, the fluctuations of the passage time at the locations of the black dots are, on the $T^{1/3}$ scale, the same as those of
their projection along the critical direction to the line $y=t/4$, the white dots.}
\label{FigureDPPext}
\end{center}
\end{figure}

(b) The results for (\ref{theta_zero_thm}) and (\ref{lpp_limit_thm_eqn}) are derived by using the knowledge of (determinantal) correlation functions. The techniques used for these models are however restricted to \emph{space-like paths} (in the best case, see~\cite{BF:2008l}), i.e., for sequences of points $(x_k,y_k)_k$ such that $x_{k+1}-x_k\geq 0$ and $y_{k+1}-y_k\leq 0$ (which can not be connected by directed paths). Now, choose in (\ref{eq13}) $\theta_k t^\nu=\tilde\theta_k t^\nu-(s_k-\tau_k^2)2^{2/3} t^{1/3}$ for some real $\tilde \theta_k$. Then, it means that we look at the height fluctuations at times $\ell(\tau_k,\theta_k,s_k)=t+\tilde \theta_k t^\nu$, with
\begin{equation}\label{eq18}
x(\tau,\theta)-y(\tau,\theta) = \tau 2^{1/3} t^{2/3},\quad x(\tau,\theta)+y(\tau,\theta) = \tfrac12 (t+\tilde\theta t^\nu)-(s-\tau^2) 2^{-1/3} t^{1/3}.
\end{equation}
Thus, one can cover much more than only the space-like regions. As before, the projection along the characteristic line of $(x,y)$ on $\tilde \theta_k=0$ leads to (\ref{eq18}) with $\tilde\theta=0$ and with a slightly modified $\tau$ (i.e., $\tau\to \tau(1+o(1))$).
Then,  using slow decorrelation, one can extend (\ref{theta_zero_thm}) to the following result: fix $m\geq 1$, $\nu\in [0,1)$, real numbers $\tau_1<\tau_2<\cdots <\tau_m$ and $s_1,\ldots, s_m$. Set
\begin{equation*}
X(\tau) = \tau 2^{1/3}t^{2/3}, \quad H(\tau,\theta,s) = \frac{1}{2}(t+\theta t^{\nu})+(\tau^2-s)2^{-1/3}t^{1/3}.
\end{equation*}
Then we have
\begin{equation}\label{gen_case}
 \lim_{t\to \infty} \PP\left(\bigcap_{k=1}^{m} \{h(X(\tau_k),t+\theta_k t^{\nu})\geq H(\tau_k,\theta_k,s_k)\}\right)
= \PP\left(\bigcap_{k=1}^{m}\{\A(\tau_k)\leq s_k\}\right).
\end{equation}
This type of computations can be readily adapted to the other KPZ models considered in the sequel.

\subsection{Point to point LPP }
Consider the following random measure $\mu$ on $\R^{d+1}$:
\begin{equation*}
\mu=\sum_{p\in \Z_+^{d+1}} w_{p}\delta_{p}
\end{equation*}
where $\Z_+=\{0,1,\ldots\}$ and $w_p$ are non-negative random variables. One may consider directed paths to be restricted to follow the lattice edges. This is just standard (point-to-point) last passage percolation (as considered for instance in~\cite{KJ:2000s}). We will restrict ourselves to the case where $d=1$, i.e., LPP in the 2-dimensional corner. Here we write $w_{i,j}$ for weights.

The conditions for our slow decorrelation theorem to hold amount to the existence of a LLN and CLT. Presently, for point to point LPP, this is only rigorously know for the
two solvable classes of weight distributions -- exponential and geometric. For general weight distributions, the existence of a LLN follows from superadditivity (via the Kingman subadditive ergodic theorem), though the exact value of the LLN is not known beyond the solvable cases. None the less, universality is expected at the level of the CLT for a very wide class of underlying weight distributions. That is to say that, after centering by the LLN, and under $t^{1/3}$ scaling, the fluctuations of LPP should always be given by the $F_{\GUE}$ distribution. 
In the results we now state we will restrict attention to exponential weights, as geometric weights lead to analogous results.

Define LPP with {\it two-sided boundary conditions} as the model with independent exponential weights such that, for positive parameters $\pi,\eta>0$,
\begin{equation}\label{eq31}
w_{i,j} = \begin{cases}
       \text{exponential of rate } \pi, &\text{if } i>0,j=0;\\
           \text{exponential of rate } \eta, &\text{if } i=0,j>0;\\
       \text{exponential of rate } 1, &\text{if } i>0,j>0;\\
           \text{zero},& \text{if } i=0,j=0.\\
          \end{cases}
\end{equation}
Recall that an exponential of rate $\lambda$ has mean $1/\lambda$. 
This class of models was introduced in~\cite{PS:2002c} (and for geometric weights in~\cite{BR:2000l}) and 
includes the one considered in~\cite{KJ:2000s}.

A full description of the one-point fluctuation limit theorems for $L_{\hl}(tp)$ was conjectured in~\cite{PS:2002c} (see Conjecture 7.1) and a complete proof was given in~\cite{BAC:2009c}. These limit theorems show that the hypotheses for slow decorrelation are satisfied and hence Theorem~\ref{growth_thm} applies. We present an adaptation of Theorem~1.3 of~\cite{BAC:2009c} stated in such a way that Theorem~\ref{growth_thm} is immediately applicable. As such, we also state the outcome of applying Theorem~\ref{growth_thm} (see Figure~\ref{slow_dec_2_sided_lpp_characteristics} for an illustration). The characteristic direction $u$ as well as the exponents and limiting distributions for the fluctuations
depend on the location of the point $p$ as well as the values of $\pi$ and $\eta$. Due to the radial scaling of $p$ it suffices to consider $p$ of the form $p=(1,\kappa^2)$ for $\kappa^2\in (0,\infty)$.

In the following theorem, $\gamma_{\pp}=1/3$, $c_{\hl}$ is a constant depending on the direction $p$, $c_{\pp}$ is a constant depending on the direction $u$, and $D'$ is $F_{\GUE}$. We also refer the reader to ~\cite{BAC:2009c} for the definitions of the distribution functions $F_{\GUE}$, $F_{\GOE}^2$ and $F_{0}$ which arise here.

\begin{theorem}\label{bac_thm}
Define $\kappa_{\eta}=\frac{\eta}{1-\eta}$, $\kappa_{\pi}=\frac{1-\pi}{\pi}$ and $\kappa_{sh}=\sqrt{\frac{\eta(1-\pi)}{\pi(1-\eta)}}$.
\begin{enumerate}
\item If $\kappa_\eta\geq \kappa \geq\kappa_\pi$ (which implies that $\pi+\eta\geq 1$) then
$\ell_{\hl}= (1+\kappa)^2$, $\gamma_{\hl} = 1/3$, and $D$ is either $F_{\GUE}$ (in the case of strict inequality) or $F_{\GOE}^2$ (in the case of either, but not both, equalities) or $F_{0}$ (in the case where all three terms in the inequality are equal). Then, there is slow decorrelation in the direction $u=p$ for all $\nu\in (0,1)$.

\item If $\pi+\eta\geq 1$ and $\kappa>\kappa_\eta$ then $\ell_{\hl}=\frac{\kappa^2}{\eta}+\frac{1}{1-\eta}$, $\gamma_{\hl}=1/2$, and $D=\mathcal{N}_{0,1}$ is the standard Gaussian distribution. Then there is slow decorrelation in the direction $u=(1, \kappa_{\eta}^2)$ for all $\nu\in (0,1)$.

\item If $\pi+\eta\geq 1$ and $\kappa<\kappa_\pi$ then $\ell_{\hl}=\frac{1}{\pi}+\frac{\kappa^2}{1-\pi}$,
$\gamma_{\hl}=1/2$, and $D=\mathcal{N}_{0,1}$. Then there is slow decorrelation in the direction $u=(1, \kappa_{\pi}^2)$ for all $\nu\in (0,1)$.

\item If $\pi+\eta<1$ and $\kappa>\kappa_{sh}$ then $\ell_{\hl}=\frac{\kappa^2}{\eta}+\frac{1}{1-\eta}$,
$\gamma_{\hl}=1/2$, and $D=\mathcal{N}_{0,1}$. Then there is slow decorrelation in the direction $u= (1,\kappa_{\eta}^2)$ for all $\nu\in (0,1)$.

\item If $\pi+\eta<1$ and $\kappa<\kappa_{sh}$  then
$\ell_{\hl}=\frac{1}{\pi}+\frac{\kappa^2}{1-\pi}$, $\gamma_{\hl}=1/2$, and $D=\mathcal{N}_{0,1}$. Then there is slow decorrelation in the direction $u= (1,\kappa_{\pi}^2)$ for all $\nu\in (0,1)$.

\item If $\pi+\eta<1$ and $\kappa=\kappa_{sh}$ then $\ell_{hl} = \frac{1}{\pi(1-\pi)}= \frac{1}{\eta(1-\eta)}$, $\gamma_{hl}=1/2$, and $D$ is distributed as the maximum of two independent Gaussian distributions. \emph{Then there is no slow decorrelation.}
\end{enumerate}
\end{theorem}
This last passage percolation model is related to a TASEP model with two-sided initial conditions
(which we discuss in Subsection~\ref{ASEP}). As explained before the characteristics are those for the Burgers equation. The first three cases above correspond with a situation that is known of as a {\it rarefaction fan}, while the last three correspond with the {\it shockwave}. The above result is illustrated in Figure~\ref{slow_dec_2_sided_lpp_characteristics}. The left case displays the rarefaction fan (the fanning of the characteristic lines from the origin) and the right case displays a shockwave (the joining together of characteristic lines coming from different directions).
\begin{figure}
\begin{center}
\psfrag{(a)}[c]{(a)}
\psfrag{(b)}[c]{(b)}
\psfrag{1}[c]{Case 1}
\psfrag{2}[c]{Case 2}
\psfrag{3}[c]{Case 3}
\psfrag{4}[c]{Case 4}
\psfrag{5}[c]{Case 5}
\psfrag{k1}[c]{$\kappa^2=\kappa^2_\eta$}
\psfrag{k2}[l]{$\kappa^2=\kappa^2_\pi$}
\psfrag{ks}[c]{$\kappa^2=\kappa^2_s$}
\includegraphics[height=6cm]{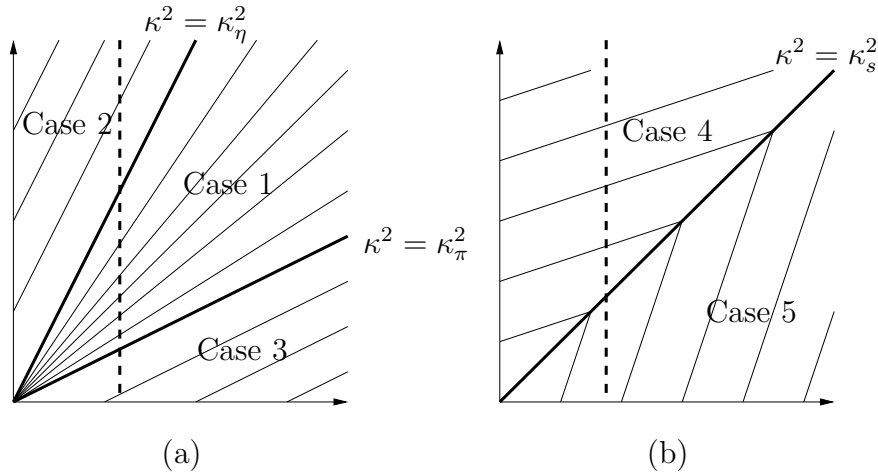}
\caption{The different cases of fluctuation limit theorems and accompanying directions of slow decorrelation:
(a) $\eta+\pi>1$ (actually shown $\eta=\pi=2/3$); (b) $\eta+\pi<1$ (actually shown $\eta=\pi=1/3$).
As $\kappa^2$ (the height along the dashed line) varies, the case of fluctuation theorem changes, as does the
direction of slow decorrelation (given by the direction of the thin lines).}
\label{slow_dec_2_sided_lpp_characteristics}
\end{center}
\end{figure}

In addition to one-point fluctuation limits, the above two-sided boundary condition LPP model has a fully classified limit process description. The description was given in~\cite{BFP:2009l} for $\pi+\eta=1$ (known as the stationary case) and in~\cite{CFP:2009l} for all other (non-equilibrium) boundary conditions. These process limits are obtained using determinantal expressions for the joint distribution of the last passage times for points along fixed directions. Thus, initially, one only gets process limits along fixed lines. As explained in Section~\ref{corner_growth_model_sec} and in~\cite{CFP:2009l} slow decorrelation, however, implies that the appropriately rescaled fluctuations at the points which are off of this line (to order $t^{\nu}$ for $\nu<1$) have the same joint distribution as their projection along characteristics to the line (see Figure 1 of~\cite{BFP:2009l} for an illustration of this).

A completely analogous situation arises in the case of geometric, rather than exponential weights (this model is often called discrete PNG). Such a model is described in~\cite{BR:2000l} and the one-point limiting fluctuation statistics are identified. The spatial process limit is characterized in~\cite{IS:2004f}. These results are only proved in a fixed space-time direction, though applying  Theorem~\ref{growth_thm} we can extend this process limit away from this fixed direction just as with the exponential weights.

A slightly different model with boundary conditions was introduced in~\cite{BBP:2005p} and involves {\it thick one-sided boundary conditions}. Fix a $k\in\N$, parameters $\pi_1,\ldots, \pi_k$, and set $\pi_i=1$ for $i>k$. Just as above, we define independent random weights on $\Z_+^2$, this time with $w_{i,j}$ exponential random variables of rate $\pi_i$ (mean $1/\pi_i$). Section 6 of~\cite{BBP:2005p} explains how results they obtain for perturbed Wishart ensembles translate into a complete fluctuation limit theorem description for this model. Just as in the two-sided boundary case, those limit theorems show that the hypotheses of Theorem~\ref{growth_thm} are satisfied and therefore there is slow decorrelation. The exponent $\gamma_{\hl}$ depends on the point $p$ and the strength of the boundary parameters $\pi_i$ and can either be $1/3$, with random matrix type fluctuations, or $1/2$ with Gaussian type (more generally $\ell \times \ell$ $\GUE$
for some $1\leq \ell \leq k$) fluctuations (see~\cite{BBP:2005p}~Theorem~1.1). The exponent $\gamma_{\pp}=1/3$ and the limiting distribution $D'$ is $F_{\GUE}$. The direction of the slow decorrelation depends on the parameters and the point (we do not write out a general parametrization of this direction as there are many cases to consider depending on the $\pi_i$). The fluctuation process limit theorem has not been proved for this model, though the method of~\cite{CFP:2009l} would certainly yield such a theorem. Also, analogous results for the geometric case have not been proved either but should be deducible from the Schur process~\cite{BP:2008a}.

\subsection{TASEP and PASEP}\label{ASEP}

\subsubsection{Totally asymmetric simple exclusion process (TASEP)}
TASEP is a Markov process in continuous time with state space $\{0,1\}^{\Z}$ (think of 1s as particles and 0s as holes). Particles jump to their right neighboring site at rate $1$, provided the site is empty. The waiting time for a jump is exponentially distributed with mean $1$ (discrete-time versions have geometrically distributed waiting times). See~\cite{TL:1999s,TL:2005i} for a rigorous construction of this process.

TASEP with different initial conditions can be readily translated into LPP with specific measures $\mu$ and hence Theorem~\ref{growth_thm} may be applied. Slow decorrelation can thus be used to show that fluctuation limit processes can be extended from fixed space-time directions to general $1+0$ dimensional space-time directions.

An observable of interest for TASEP is the integrated current of particles $I(x,t)$ defined as the number of particles which jumped from $x$ to $x+1$ during the time interval $[0,t]$. Also of interest is the height function $h(x,t)$
\begin{equation}\label{height_function_def}
 h(x,t)=\begin{cases}
 2I(0,t)+\sum_{i=1}^{x}(1-2\eta_{t}(i)), &x\geq 1,\\
 2I(0,t), & x=0,\\
 2I(0,t)-\sum_{i=x+1}^{0}(1-2\eta_{t}(i)), &x\leq -1,
 \end{cases}
\end{equation}
where $\eta_t(i)=1$ (resp.\ $\eta_t(i)=0$) is site $i$ is occupied (resp.\ empty) at time $t$. There is a simple relationship between the current and the height function given by
\begin{equation}\label{height_current}
I(x,t)=\tfrac12 (h(x,t)-x).
\end{equation}

A well-studied initial condition is {\it step initial condition}: At time $t=0$, $\{\ldots, -2,-1,0\}$ is filled by particles and $\{1,2,\ldots\}$ is empty, i.e., $h(x,t=0)=|x|$. There is a simple relation with the corner growth model studied in Subsection~\ref{corner_growth_model_sec}. The weight $w_{i,j}$ is the waiting time (counted from the instant when the particle can jump) for the particle which started in position $-j+1$ to move from position $i-j$ to $i-j+1$. Thus, the TASEP height function records the boundary of the region of points $p$ for which $L_{\hl}(p)\leq t$. Therefore, as in Subsection~\ref{corner_growth_model_sec}, one can apply Theorem~\ref{growth_thm} leading to slow decorrelation (in the sense of equation (\ref{eq11})) for the
fluctuations of the TASEP height function along space-time lines corresponding to the characteristics of
Burgers equation.

\begin{figure}
\begin{center}
\psfrag{(a)}[c]{(a)}
\psfrag{(b)}[c]{(b)}
\psfrag{x}[c]{$x$}
\psfrag{t}[c]{$t$}
\includegraphics[height=4.5cm]{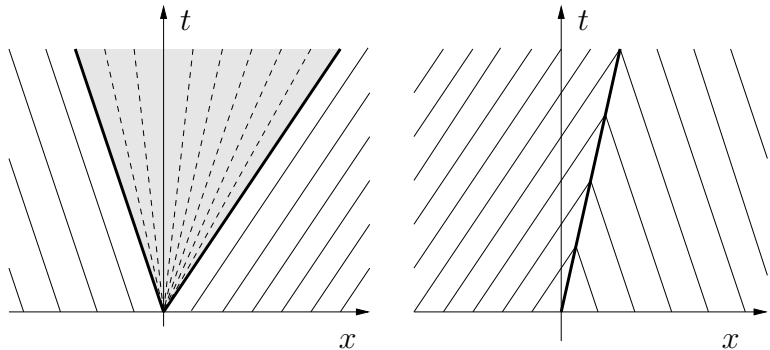}
\caption{Rarefaction wave on the left and shockwave on the right.}
\label{slow_dec_2_sided_bc}
\end{center}
\end{figure}

An important generalization of the step initial condition are the {\it two-sided Bernoulli initial conditions} which are defined for all pairs $\rho_-,\rho_+\in [0,1]$ as the random initial conditions in which particles initially occupy sites to the left of the origin with probability $\rho_-$ (independently of other sites) and likewise to the right of the origin with probability $\rho_+$. It was proven in~\cite{PS:2002s} that two-sided TASEP can be mapped\footnote{The mapping requires a geometric number of zero weights along the boundary which do not affect asymptotics.} to the LPP with two-sided boundary conditions model (\ref{eq31})  with $\pi=1-\rho_+$ and $\eta=\rho_-$. Using this connection and slow decorrelation, one can show that all the results stated for the LPP model (\ref{eq31}) can be translated into their counterpart for two-sided TASEP.
This is made in detail in~\cite{CFP:2009l} (which uses some arguments of this paper), where we prove a
complete fluctuation process limit for $\rho_-\neq \rho_+$ which complements the recent result of~\cite{BFP:2009l} for $\rho_-=\rho_+$.

The characteristic line leaving position $x$ has slope $1-2\rho(x)$. On top of this, the entropy condition ensures that if $\rho_->\rho_+$, there will be a rarefaction fan from the origin which will fill the void between lines of slope $1-2\rho_-$ and $1-2\rho_+$. The Rankine-Hugoniot condition applies to the case where $\rho_-<\rho_+$ and introduces shockwaves with specified velocities when characteristic lines would cross. These two types of characteristics are illustrated side-by-side in Figure~\ref{slow_dec_2_sided_bc}.

Another variation is TASEP with {\it slow particles or slow start-up times}, which is considered in~\cite{JB:2006p}. It may likewise be connected to the LPP with thick one-sided boundary conditions model which we previously introduced. As a result we may similarly conclude slow decorrelation.

Not all initial conditions correspond to LPP with weights restricted to $\Z_+^2$. For example, TASEP
with {\it flat (or periodic) initial conditions} corresponds to the case where only the sites of $k\Z$, for $k\geq 2$ are initially occupied. For simplicity, we focus on the case $k=2$. Then, the height function at time $t=0$ is a saw-tooth, see Figure~\ref{slow_dec_flat_tasep}(a) (though asymptotically flat, from which the name). Rotating by $\pi/2$, it is the growth interface for half-line to point LPP where the measure $\mu$ is supported on points in $(i,j)\in\Z^2$ such that $i+j\geq 0$ and given by delta masses with independent exponential weights of rate $1$. Fluctuation theorems and limit process have been proved for several periodic initial conditions~\cite{BFPS:2007f,BFP:2007f} (in~\cite{BFP:2007f} was in discrete time, i.e., with geometric weights).

Similarly TASEP with {\it half-flat initial conditions} is defined by letting particles start at $2\Z_-=\{\cdots, -4,-2,0\}$. The corresponding last passage percolation model has non-zero weights for points $(i,j)$ such that $i+j\geq 0$ and $j\geq 0$. The limit process for this model was identified in~\cite{BFS:2007t}. Theorem~\ref{growth_thm} applies to both of these model and proves slow decorrelation. This implies that the fluctuation process limits extend to general $1+0$ dimensional space-time directions. The characteristics lines are shown in Figure~\ref{slow_dec_flat_tasep}(b).
\begin{figure}
\begin{center}
\psfrag{a}[c]{(a)}
\psfrag{b}[c]{(b)}
\includegraphics[height=4.5cm]{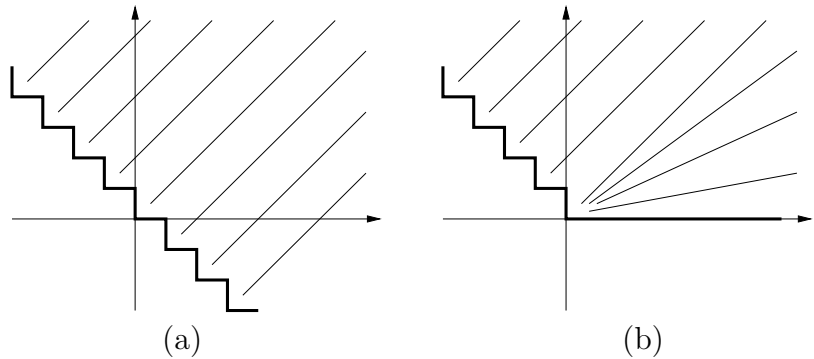}
\caption{Flat (a) and half-flat (b) TASEP correspond, via their height functions, with the LPP models in the
two regions shown above. Characteristic lines are perpendicular to the initial direction of the height function
and in case (b) the entropy condition implies that they fan out to the right of the origin.}
\label{slow_dec_flat_tasep}
\end{center}
\end{figure}

A variant of half flat initial conditions has particles starting at $2\Z_-$ plus a few particles at positive
even integers, with a different speed $\alpha$. This is known as {\it two speed} TASEP and~\cite{BPS:2009t}
gives a complete description and proof of the process limit for these initial conditions. As with all of the other examples, this can be coupled with a LPP model and hence Theorem~\ref{growth_thm} applies and prove slow decorrelation and enables us to extend these process limit results as well.

\subsubsection{Partially asymmetric simple exclusion process (PASEP)}
The PASEP is a generalization of TASEP where, particles jump to the right-neighboring site with
rate $p\neq 1/2$ and to the left-neighboring site with rate $q=1-p$ (always provided that the destination
sites are empty). An important tool to study PASEP is the {\it basic coupling}~\cite{TL:1999s,TL:2005i}.
Through a graphical construction, one can realize and hence couple together every PASEP (with different initial conditions) on the same probability space. Even though PASEP 
can not be mapped to a LPP model, it still has the same super-additivity properties necessary to prove a version of Theorem~\ref{growth_thm}. The property comes in the form of {\it attractiveness}. That PASEP is attractive means that if you start with two initial conditions corresponding to height functions $h_1(x,0)\leq h_2(x,0)$ for all $x\in \R$, then for any future time $t$, $h_1(x,t)\leq h_2(x,t)$ for all $x\in \R$.

We now briefly review this graphical construction. Above every integer draw a half-infinite time ladder.
Fix $p$ (and hence $q$) and for each ladder place right and left horizontal arrows independently at Poisson
points with rates $p$ and $q$ respectively. This is the common environment in which all initial conditions
may be coupled. Particles move upwards in time until they encounter an arrow leaving their ladder. They
try to follow this ladder, and hence hop one step, yet this move is excluded if there is already another particle on the neighboring ladder. That this graphical construction leads to attractiveness for the PASEP is shown, for instance, in ~\cite{TL:1999s,TL:2005i}.

In a series of three papers~\cite{TW:2008i,TW:2008f,TW:2008a} Tracy and Widom show that for step initial conditions with positive drift $\gamma=p-q>0$, PASEP behaves asymptotically the same as TASEP (when speeded-up by $1/\gamma$). Just as in TASEP the current or height function is of central interest. $I(x,t)$ is defined as the number of particles which jumped from $x$ to $x+1$ minus the ones from $x+1$ to $x$ during $[0,t]$ and
$h(x,t)$ is defined by (\ref{height_function_def}). This time, the height function does not monotonically grow, but does still have a drift.

The slow decorrelation theorem for PASEP with general initial conditions is stated below.
By a PASEP model we mean a measure on initial configurations, as well as a rate $p=1-q\in (1/2,1]$.
We write $h(x,t)$ for the height function for this specified model, and $h'(x,t)$ for the height function for
the PASEP with step initial conditions. Note that the generalizations of Remark~\ref{non_const_remark} apply in this case too.

\begin{theorem}[Slow decorrelation for PASEP]\label{ASEP_thm}
Consider a velocity $v\in \R$ and a second variable $u\in \R$.
If there exist constants (depending on $v$ and $u$ and the model): $\ell$ and $\ell'$ non-negative;
$\gamma, \gamma'\in (0,1)$; $\nu\in (0,\gamma/\gamma')$; and distributions $D$ and $D'$ such that
\begin{eqnarray}
\nonumber &&\frac{h(vt,t)-t\ell}{t^{\gamma}}\Longrightarrow D, \quad \textrm{as }t\textrm{ goes to infinity},\\
&&\frac{h(vt+ut^{\nu},t+t^{\nu})-t\ell-t^\nu \ell'}{t^{\gamma}}\Longrightarrow D, \quad \textrm{as }t\textrm{ goes to infinity},\\
\nonumber &&\frac{h'(ut,t)-t\ell'}{t^{\gamma'}}\Longrightarrow D', \quad \textrm{as }t\textrm{ goes to infinity},
\end{eqnarray}
then we have slow decorrelation of the PASEP height function at speed $v$, in the direction given by $u$ and with scaling exponent $\nu$, i.e., for all  $M>0$,
\begin{equation*}
\lim_{t\to \infty} \PP(|h(vt+ut^{\nu},t+t^{\nu})-h(vt,t)-t^{\nu}\ell'|\geq M t^{\gamma})=0.
\end{equation*}
\end{theorem}

\begin{proof}[Proof of Theorem~\ref{ASEP_thm}]
Rather than the height function we focus on the current which is related via equation (\ref{height_current}).
$I(vt+ut^{\nu},t+t^{\nu})$ is equal to the current $I(vt,t)$ up to time $t$, plus the current of particles which cross the space-time line from $vt$ at time $t$ to $vt+ut^{\nu}$ at time $t+t^{\nu}$. We consider a coupled system starting at time $t$ reset so as to appear to be in step initial conditions centered at position $vt$. By attractiveness of the basic coupling, the current across the space-time line from $vt$ at time $t$ to $vt+ut^{\nu}$ at time $t+t^{\nu}$
for this ``step'' system will exceed that for the original system.
Denote by $I'(ut^{\nu},t^{\nu})$ the current associated to the coupled ``step'' system and observe that, it is distributed as the current of an independent step initial condition PASEP.
Thus,
\begin{equation*}
I(vt+ut^{\nu},t+t^{\nu})= I(vt,t)+I'(ut^{\nu},t^{\nu})+X_t
\end{equation*}
where $X_t\leq 0$. From this point on the proof follows exactly as in the proof of Theorem~\ref{growth_thm}.
\end{proof}

Using the fluctuation results proved in~\cite{TW:2008i,TW:2008f,TW:2008a}, reviewed in~\cite{TW:2009t},
we find that the above hypotheses are satisfied for PASEP with step initial conditions and also for step Bernoulli initial conditions~\cite{TW:2009o} with $\rho_-=0$ and $\rho_+>0$. The slow decorrelation directions are given by the characteristics just as in the case of TASEP. 
These two sets of initial conditions are the only ones for which fluctuations theorems are presently known for PASEP, but limit process theorems are not yet proven.

\subsection{The polynuclear growth (PNG) model}\label{PNG}
As mentioned before, slow decorrelation for the (continuous time, Poisson point) PNG model was proved previously in~\cite{PLF:2008s} in the case of the {\it PNG droplet} and {\it stationary PNG}. Theorem~\ref{growth_thm} (along with the necessary preexisting fluctuation theorems) gives an alternative proof of these results as well as the analogous result for {\it flat PNG}. Because of the minimality of the hypotheses of our theorem we may further prove slow decorrelation for the model of {\it PNG with two (constant) external sources} considered in~\cite{BR:2000l}. The way that PNG fits into the framework of our half-line to point LPP model is that one takes $\mu$ to be a Poisson point process of specified intensity on some domain. For the PNG droplet, stationary PNG and PNG with two external sources, we restrict the point process to $\R_+^2$ and (in the second and third cases) augment the measure $\mu$ with additional one dimensional point process along the boundaries. For flat PNG the support of the point process is
$\{(x,y):x+y\geq 0\}$. The limit process for the PNG droplet for fixed time was proved in~\cite{PS:2002s}
and for flat PNG was proved in~\cite{BFS:2008l} for space-like paths. It was explained in~\cite{PLF:2008s}
that slow decorrelation implies that these limit processes extend to general $1+0$ dimensional space-time
directions (with time scaling $t^{\nu}$ for $\nu<1$).

\subsection{First passage percolation}\label{FPP}
As opposed to LPP one can look to the minimum value of $T(\pi)$. This then goes by the name of directed {\it first} passage percolation and for simplicity we consider this only when we restrict our measure to being supported on a lattice. One may also consider undirected first passage percolation. Theorem~\ref{growth_thm} can be adapted in a straightforward way for both of these models. The statement of the theorem remains identical up to replacing the last passage time variable with the first passage time. For the proof the only change is that the compensator $X_t$ now satisfies $X_t\leq 0$ rather than $X_t\geq 0$.
Unfortunately no fluctuation theorems have been proved for first passage percolation, so all that we get is a criterion for slow decorrelation.

\subsection{Directed polymers}\label{DP}
We now briefly consider a lattice-based directed polymer models in $1+1$ dimension and note that just as in LPP, slow decorrelation can arise in these models. Unfortunately, just as in first passage percolation, there are no fluctuation theorems proved for such polymers. Recently, however, the order of fluctuations for a particular specialization of this model was proved in~\cite{TS:2009s}. It should be noted that while we focus on just one model, the methods used can be applied to other polymer models and in more than $1+1$ dimension (for example line to point polymers).

The model we consider is the point to point directed polymer. In this model we consider any directed, lattice
path $\pi$ from $(0,0)$ to a point $p$ and assign it a Gibbs weight $e^{\beta T(\pi)}$ where $\beta\geq 0$ is
known as the inverse temperature and where $T(\pi)$ is the sum of weights (which are independent) along the
path $\pi$ ($-T(\pi)$ is the energy of the path $\pi$). We define the partition function and free energy for a polymer from a point $p$ to $q$ as:
\begin{equation*}
 Z_{\beta}(p,q)=\sum_{\pi:p\to q} e^{\beta T(\pi)}, \qquad \fr_\beta(0,p)=\frac{1}{\beta}\log Z_\beta(0,p).
\end{equation*}
It is expected that the free energy satisfies similar fluctuation theorems to those of LPP (which is the $\beta=\infty$ limit of $\fr_\beta(0,p)$).

\begin{theorem}\label{polymer_thm}
Consider a directed polymer model and consider a point $p\in \R_+^{2}$ and a direction $u\in\R_+^{2} $.
If there exist constants (depending on $p$ and $u$ and the model weight distributions): $\ell$ and $\ell'$
non-negative; $\gamma,\gamma'\in (0,1)$; $\nu\in(0,\gamma/\gamma')$; and distributions $D$, $D'$ such that
\begin{equation*}
\begin{aligned}
\chi_1(t)&:=\frac{\fr_\beta(0,tp)-t\ell}{t^{\gamma}}\Longrightarrow D, \quad \textrm{as }t\textrm{ goes to infinity},\\
\chi_2(t)&:=\frac{\fr_\beta(0,tp+t^{\nu}u)-t\ell-t^\nu \ell'}{t^{\gamma}}\Longrightarrow D, \quad \textrm{as }t\textrm{ goes to infinity},\\
\chi_3(t)&:=\frac{\fr_\beta(tp,tp+t^{\nu}u)-t^\nu \ell'}{(t^{\nu})^{\gamma'}}\Longrightarrow D', \quad \textrm{as }t\textrm{ goes to infinity},
\end{aligned}
\end{equation*}
then we have slow decorrelation of the point to point polymer at $tp$, in the direction $u$ and with scaling
exponent $\nu$, which is to say that for all  $M>0$,
\begin{equation*}
\lim_{t\to \infty} \PP(|\fr_\beta(0,tp+t^{\nu}u)-\fr_\beta(0,tp)-t^\nu \ell'|\geq M t^{\gamma})=0.
\end{equation*}
\end{theorem}
\begin{proof}
The direction $u$ for a given $p$ should correspond to the characteristic through that point. The proof of this criterion for slow decorrelation is identical to the proof for Theorem~\ref{growth_thm} and follows from the computation below (a result of super-additivity yet again):
\begin{equation*}
\begin{aligned}
\fr_\beta(0,tp+t^{\nu}u) &=
 \frac{1}{\beta}\log \bigg(\sum_{\pi:0\to tp} e^{\beta T(\pi)} \times \sum_{\pi:tp\to tp+t^{\nu}u} e^{\beta T(\pi)}
 + \sum_{\substack{\pi:0\to tp+t^{\nu}u,\\ tp\notin \pi}} e^{\beta T(\pi)}\bigg)\\
&= \frac{1}{\beta}\log \bigg(\sum_{\pi:0\to tp} e^{\beta T(\pi)} \times \sum_{\pi:tp\to tp+t^{\nu}u} e^{\beta T(\pi)}\bigg) + X_t\\
&= \fr_\beta(0,tp) + \fr_\beta(tp,tp+t^{\nu}u) + X_t.
\end{aligned}
\end{equation*}
Here  $X_t\geq 0$ and the argument is analogous to (\ref{growth_compensator}).
\end{proof}

\bibliographystyle{alpha}

\begin{thebibliography}{99}

\bibitem{ACQ:2010p}
G.~Amir, I.~Corwin, and J.~Quastel.
\newblock Probability distribution of the free energy of the continuum directed random polymer in 1+1 dimensions.
\newblock {\em Comm. Pure Appl. Math}, 64:466--537, 2011.

\bibitem{JB:2006p}
J.~Baik.
\newblock Painlev\'{e} formulas of the limiting distributions for nonnull complex sample covariance matrices.
\newblock {\em Duke Math. Journal}, 33:205--235, 2006.

\bibitem{BBP:2005p}
J.~Baik, G.~Ben~Arous, and S.~P\'{e}ch\'{e}.
\newblock Phase transition of the largest eigenvalue for nonnull complex sample covariance matrices.
\newblock {\em Ann. Probab.}, 33:1643--1697, 2005.

\bibitem{BFP:2009l}
J.~Baik, P.L.~Ferrari, and S.~P\'{e}ch\'{e}.
\newblock Limit process of stationary TASEP near the characteristic line.
\newblock arxiv:0907.0226; {\em Comm. Pure Appl. Math.}, To appear.

\bibitem{BR:2000l}
J.~Baik and E.~Rains.
\newblock Limiting distributions for a polynuclear growth model with external sources.
\newblock {\em J. Stat. Phys.}, 100:523--542, 2000.

\bibitem{BAC:2009c}
G.~Ben~Arous and I.~Corwin.
\newblock Current Fluctuations for TASEP: A Proof of the Pr\"{a}hofer-Spohn Conjecture.
\newblock {\em Ann. Probab.}, 39:104--138, 2011.

\bibitem{BG:1997s}
L.~Bertini, G.~Giacomin.
\newblock  Stochastic Burgers and KPZ equations from particle systems.
\newblock {\em Comm. Math. Phys.} 183:571--607, 1997.

\bibitem{BF:2008l}
A.~Borodin and P.L.~Ferrari.
\newblock Large time asymptotics of growth models on space-like paths I: PushASEP.
\newblock {\em Electron. J Probab.}, 13:1380--1418, 2008.

\bibitem{BFP:2007f}
A.~Borodin, P.L.~Ferrari and M.~Pr\"{a}hofer.
\newblock Fluctuations in the discrete TASEP with periodic initial configurations and the $\textrm{Airy}_1$ process.
\newblock {\em Int. Math. Res. Papers}, rpm002, 2007.

\bibitem{BFPS:2007f}
A.~Borodin, P.L.~Ferrari, M.~Pr\"{a}hofer and T.~Sasamoto.
\newblock Fluctuation properties of the TASEP with periodic initial configuration.
\newblock {\em J. Stat. Phys.}, 129:1055--1080, 2007.

\bibitem{BFS:2007t}
A.~Borodin, P.L.~Ferrari and T.~Sasamoto.
\newblock Transition between $\textrm{Airy}_1$ and $\textrm{Airy}_2$ processes and TASEP fluctuations.
\newblock {\em Comm. Pure Appl. Math.}, 61:1603--1629, 2007.

\bibitem{BFS:2008l}
A.~Borodin, P.L.~Ferrari and T.~Sasamoto.
\newblock Large Time Asymptotics of Growth Models on Space-like Paths II: PNG and Parallel TASEP.
\newblock {\em Comm. Math. Phys.}, 283:417--449, 2008.

\bibitem{BPS:2009t}
A.~Borodin, P.L. Ferrari and T.~Sasamoto.
\newblock Two speed TASEP.
\newblock {\em J. Stat. Phys.}, 1572--9613, 2009.

\bibitem{BP:2008a}
A.~Borodin and S.~P\'{e}ch\'{e}.
\newblock Airy kernel with two sets of parameters in directed percolation and random matrix theory.
\newblock {\em J. Stat. Phys.}, 132:275--290, 2008.

\bibitem{CFP:2009l}
I.~Corwin, P.L.~Ferrari and S.~P\'{e}ch\'{e}.
\newblock Limit processes for TASEP with shocks and rarefaction fans.
\newblock {\em J. Stat. Phys.}, 140:232--267, 2010.


\bibitem{CQ:2010u}
I.~Corwin, and J.~Quastel.
\newblock Universal distribution of fluctuations at the edge of the rarefaction fan.
\newblock arXiv:1006.1338.

\bibitem{CQ:2011r}
I.~Corwin, and J.~Quastel.
\newblock Renormalization fixed point of the KPZ universality class.
\newblock In preparation.

\bibitem{E:1998p}
L.C.~Evans.
\newblock {\em Partial Differential Equations}.
\newblock AMS, Providence, 1998.

\bibitem{PLF:2008s}
P.~L. Ferrari.
\newblock Slow decorrelations in KPZ growth.
\newblock {\em J. Stat. Mech.}, P07022, 2008.

\bibitem{Fer10b}
P.~L. Ferrari.
\newblock From interacting particle systems to random matrices.
\newblock {\em J. Stat. Mech.}, P10016, 2010.

\bibitem{IS:2004f}
T.~Imamura and T.~Sasamoto.
\newblock Fluctuations of the one-dimensional polynuclear growth model with
  external sources.
\newblock {\em Nucl. Phys. B}, 699:503--544, 2004.

\bibitem{SI07}
 T. Imamura and T. Sasamoto.
\newblock Dynamical properties of a tagged particle in the totally asymmetric simple exclusion process with the step-type initial condition.
\newblock {\em J. Stat. Phys.}, 128:799--846, 2007.

\bibitem{KJ:2000s}
K.~Johansson.
\newblock Shape fluctuations and random matrices.
\newblock {\em Comm. Math. Phys.}, 209:437--476, 2000.

\bibitem{KJ:2003d}
K.~Johansson.
\newblock Discrete polynuclear growth and determinantal processes.
\newblock {\em Comm. Math. Phys.}, 242:277--329, 2003.

\bibitem{Krug99}
H.~Kallabis and J.~Krug.
\newblock Persistence of Kardar-Parisi-Zhang interfaces.
\newblock {\em Europhys. Lett.}, 45:20--25, 1999.

\bibitem{KPZ:1986d}
M.~Kardar, G.~Parisi, Y.C~Zhang.
\newblock Dynamic scaling of growth interfaces.
\newblock {\em Phys. Rev. Let.}, 56:889--892, 1986.

\bibitem{Maj97}
J.~Krug, H.~Kallabis, S.N.~Majumdar, S.J.~Cornell, A.J.~Bray and C.~Sire.
\newblock Persistence exponents for fluctuating interfaces.
\newblock {\em Phys. Rev. E}, 56:2702, 1997.

\bibitem{Spo95}
J.~Krug and H.~Spohn.
\newblock Kinetic roughening of growning surfaces.
\newblock {\em Solids Far From Equilibrium}, Godr\`eche ed., Cambridge Univ. Press, 1991.

\bibitem{TL:1999s}
T.M. Liggett.
\newblock {\em Stochastic Interacting Systems: Contact, Voter and Exclusion Processes}.
\newblock Springer, Berlin, 1999.

\bibitem{TL:2005i}
T.M. Liggett.
\newblock {\em Interacting Particle Systems}.
\newblock Springer, Berlin, reprint of 1985 original edition, 2005.

\bibitem{PS:2002c}
M.~Pr\"{a}hofer and H.~Spohn.
\newblock Current fluctuations for the totally asymmetric simple exclusion process.
\newblock {\em Progress in Probability}, 51:185--204, 2002.

\bibitem{PS:2002s}
M.~Pr\"{a}hofer and H.~Spohn.
\newblock Scale invariance of the PNG droplet and the Airy process.
\newblock {\em J. Stat. Phys.}, 108:1071--1106, 2002.

\bibitem{HR:1981n}
H.~Rost.
\newblock Non-equilibrium behavior of a many particle process: Density profile and the local equilibrium.
\newblock {\em Z. Wahrsch. Verw. Gebiete}, 58:41--53, 1981.

\bibitem{SS:2010u}
T.~Sasamoto and H.~Spohn.
\newblock Universality of the one-dimensional KPZ equation.
\newblock {\em Phys. Rev. Lett.}, 104:230602, 2010.

\bibitem{TS:2009s}
T.~Sepp\"{a}l\"{a}nen.
\newblock Scaling for a one-dimensional directed polymer with constrained endpoints.
\newblock arXiv:0911.2446

\bibitem{TS:1998h}
T.~Sepp\"{a}l\"{a}inen.
\newblock Hydrodyanamic scaling, convex duality and asymptotic shapes of growth models.
\newblock {\em Markov Proc. Relat. Fields} 4:1--26, 1998.

\bibitem{TW:1994l}
C.~Tracy and H.~Widom.
\newblock Level-spacing distributions and the Airy kernel.
\newblock {\em Comm. Math. Phys.}, 159:151--174, 1994.

\bibitem{TW:2008i}
C.~Tracy and H.~Widom.
\newblock Integral formulas for the asymmetric simple exclusion process.
\newblock {\em Comm. Math. Phys.}, 279:815--844, 2008.

\bibitem{TW:2008f}
C.~Tracy and H.~Widom.
\newblock A Fredholm determinant representation in ASEP.
\newblock {\em J. Stat. Phys.}, 132:291--300, 2008.

\bibitem{TW:2008a}
C.~Tracy and H.~Widom.
\newblock Asymptotics in ASEP with step initial condition.
\newblock {\em Comm. Math. Phys.} 290:129--154, 2009.

\bibitem{TW:2009t}
C.~Tracy and H.~Widom.
\newblock Total current fluctuations in ASEP.
\newblock {\em J. Math. Phys.} 50,095204, 2009.

\bibitem{TW:2009o}
C.~Tracy and H.~Widom.
\newblock On ASEP with Step Bernoulli Initial Condition.
\newblock {\em J. Stat. Phys.} 137:291--300, 2008.

\end{thebibliography}

\end{document}